%% file: ms.tex
\documentclass[11pt]{article}  
\input{Headers/Header}

\input{Headers/Notation}

\title{
Unbiased Shrinkage Estimation}
\author{Jann Spiess}
\date{\textsc{Working Paper} \\ This version: October 31, 2017}

\begin{document}
    
    \maketitle
    
    \begin{abstract}
        Shrinkage estimation usually reduces variance at the cost of bias. But when we care only about some parameters of a model, I show that we can reduce variance without incurring bias if we have additional information about the distribution of covariates. In a linear regression model with homoscedastic Normal noise, I consider shrinkage estimation of the nuisance parameters associated with control variables. For at least three control variables and exogenous treatment, I establish that the standard least-squares estimator is dominated with respect to squared-error loss in the treatment effect even among unbiased estimators and even when the target parameter is low-dimensional. I construct the dominating estimator by a variant of James–Stein shrinkage in a high-dimensional Normal-means problem. It can be interpreted as an invariant generalized Bayes estimator with an uninformative (improper) Jeffreys prior in the target parameter.
    \end{abstract}
    
    \blfootnote{\hspace{-\baselineskip}
        Jann Spiess, Department of Economics, Harvard University, \href{mailto:jspiess@fas.harvard.edu}{\texttt{jspiess@fas.harvard.edu}}.
        I thank Gary Chamberlain, Maximilian Kasy, Carl Morris, and Jim Stock for insightful conversations, and seminar participants at Harvard for helpful comments.
    }%

    \section*{Introduction}
    
    Many inference tasks have the following feature: the researcher wants to obtain a high-quality estimate of a set of target parameters (for example, a set of treatment effects in an RCT), but also estimates a number of nuisance parameters she does not care about separately (for example, coefficients on control variables).
    In these cases, can we reduce variance in the estimation of a target parameter without inducing bias by shrinking in the estimation of possibly high-dimensional nuisance parameters?
    In a linear regression model with homoscedastic, Normal noise, I show that a natural application of James--Stein shrinkage to the parameters associated with at least three control variables reduces loss in the possibly low-dimensional treatment effect parameter without producing bias provided that treatment is random.
    
    The proposed estimator effectively averages between regression models with and without control variables,
    similar to the \cite{Hansen:2016ix} model-averaging estimator
    and coinciding up to a degrees-of-freedom correction with the corresponding Mallows estimator from \cite{Hansen:2007cqa}.
    For the specific choice of shrinkage, I contribute three finite-sample properties:
    First, I note that by averaging over the distribution of controls we obtain dominance of the shrinkage estimator even for low-dimensional target parameters, unlike other available results that require a loss function that is at least three-dimensional.
    Second, I establish that the resulting estimator remains unbiased under exogeneity of treatment.
    Third,
    I conceptualize it as a two-step estimator with a first-stage prediction component.
    Fourth, I show that it can be seen as a natural, invariant generalized Bayes estimator with respect to a partially improper prior corresponding to uninformativeness in the target parameter.
    
    The linear regression model is set up in Section~\ref{sect:setup}.
    Section~\ref{sect:control} proposes the estimator and establishes loss improvement relative to a benchmark OLS estimator provided treatment is exogenous.
    Section~\ref{sect:invariance} motivates the estimator as an invariant generalized Bayes estimator (with respect to an improper prior) in a suitably transformed many-means problem.

    \section{Linear Regression Setup}
    \label{sect:setup}
    
    I consider estimation of the structural parameter $\beta \in \R^k$ in
    the canonical linear regression model
    \begin{align}
    \label{eqn:normal}
        Y_i = \alpha + X_i' \beta + W_i' \gamma + U_i
    \end{align}
    from $n$ iid observations $(Y_i,X_i,W_i)$, where $X_i \in \R^m$ are the regressors of interest, $W_i \in \R^k$ control variables, and $U_i \in \R$ is homoscedastic, Normal noise.
    $\alpha$ is an intercept,%
    \footnote{We could alternatively include a constant regressor in $X_i$ and subsume $\alpha$ in $\beta$.
    I choose to treat $\alpha$ separately since I will focus on the loss in estimating $\beta$, ignoring the performance in recovering the intercept $\alpha$.} 
    and $\gamma$ is a nuisance parameter.
    To obtain identification of $\beta$ in \Cref{eqn:normal},
    I assume that $U_i$ is orthogonal to $X_i$ and $W_i$ (no omitted variables).
    
    Throughout this document, I write upper-case letters for random variables (such as $Y_i$) and lower-case letters for fixed values (such as when I condition on $X_i = x_i$).
    When I suppress indices, I refer to the associated vector or matrix of observations, e.g. $Y \in \R^n$ is the vector of outcome variables $Y_i$ and $X \in \R^{n \times m}$ is the matrix with rows $X'_i$.

    \section{Two-Step Partial Shrinkage Estimator}
    \label{sect:control}
    
    By assumption there are control variables $W$ available with
    \begin{align*}
        Y | X{=}x, W{=}w &\sim \N(\1 \alpha + x \beta + w \gamma, \sigma^2 \I_n)
    \end{align*}
    where $\sigma^2$ need not be known.
    We care about the (possibly high-dimensional) nuisance parameter $\gamma$ only in so far as it helps us to estimate the (typically low-dimensional) target parameter $\beta$, which is our object of interest.
        
    \subsection{A canonical form that preserves structure}
    
    Given $x \in \R^{n \times m}$ and $w \in \R^{n \times k}$,
    where we assume that $(\1,x,w)$ has full rank $1 + m + k \leq n$,
    let $q = (q_{\1},q_x,q_w,q_r) \in R^{n \times n}$ orthonormal
    where $q_{\1} \in \R^n, q_x \in \R^{n \times m}, q_w \in \R^{n \times k}$
    such that $\1$ is in the linear subspace of $\R^n$ spanned by $q_{\1} \in \R^{n}$
    (that is, $q_{\1} \in \{\1/\sqrt{n},-\1/\sqrt{n}\}$),
    the columns of $(\1,x)$ are in the space spanned by the columns of $(q_{\1},q_x)$,
    and the columns of $(\1,x,w)$ are in the space spanned by the columns of $(q_{\1},q_x,q_w)$.
    (Such a basis exists, for example, by an iterated singular value decomposition.)
    Then,
    \begin{align*}
        Y^* = q' Y | X{=}x, W{=}w &\sim \N\left(
        \begin{pmatrix}
            q'_{\1} \1 \alpha + q'_{\1} x \beta + q'_{\1} w \gamma \\
            q'_{x} x \beta + q'_{x} w \gamma \\
            q'_{w} w \gamma \\
            \0_{n - 1 - m - k}
        \end{pmatrix},
        \sigma^2 \I_n
        \right).
    \end{align*}
    Writing $Y^*_x$, $Y^*_w, Y^*_r$ for the appropriate subvectors of $Y^*$,
    we find, in particular, that
    \begin{align*}
        \begin{pmatrix}
            Y^*_x \\
            Y^*_w \\
            Y^*_r
        \end{pmatrix}
        | X{=}x, W{=}w 
        \sim
        \N\left(
            \begin{pmatrix}
                \mu_x + a \mu_w \\
                \mu_w \\
                \0_{n-1-m-k}
            \end{pmatrix},
            \sigma^2 \I_{n - 1}
        \right)
    \end{align*}
    where $\mu_x = q'_{x} x \beta \in \R^m$, $\mu_w = q'_{w} w \gamma \in \R^k$, and $a = q'_x w (q'_w w)^{-1} \in \R^{m \times k}$.%
    \footnote{Alternatively, we could have denoted by $\mu_x$ the mean of $Y_x^*$.
    However, by separating out $\mu_x$ from $a \mu_w$ I feel that the role of $\mu_w$ as a relevant nuisance parameter becomes more transparent.}
    In transforming linear regression to this Normal-means problem, as well as in partitioning the coefficient vector into two groups, for only one of which I will propose shrinkage, I follow \cite{Sclove:1968ja}.
    
    \subsection{Two-step estimator}
    \label{subsect:twostep}
    
    Conditional on $X{=}x, W{=}w$ and given an estimator $\hat{\mu}_w = \hat{\mu}_w(Y^*_w,Y^*_r)$ of $\mu_w$,
    a natural estimator of $\mu_x$ is
    $
        \hat{\mu}_x
        = \hat{\mu}_x(Y^*_x,Y^*_w,Y^*_r)
        = Y^*_x - a \hat{\mu}_w
    $.
    An estimator of $\beta$ is obtained by setting $\hat{\beta} = (q'_x x)^{-1} \hat{\mu}_x$.
    (The linear least-squares estimator for $\beta$ is obtained from $\hat{\mu}_w = Y^*_w$.)
    A natural loss function for $\hat{\beta}$ that represents prediction loss units is the weighted loss $(\hat{\beta} - \beta)' (x' q_x q'_x x) (\hat{\beta} - \beta) = \| \hat{\mu}_x - \mu_x \|^2$.
    We can therefore focus on the (conditional) expected squared-error loss in estimating $\mu_x$, for which we find
    \begin{align*}
        \E[\| \hat{\mu}_x - \mu_x \|^2 |X{=}x,W{=}w]
        =
        m \sigma^2
        + \E[\| \hat{\mu}_w - \mu_w \|_{a'a}^2 |X{=}x,W{=}w]
    \end{align*}
    with the seminorm $\| v \|_{a'a} = \sqrt{v' a'a v}$ on $\R^k$.
    
    For high-dimensional $\mu_w$ ($k \geq 3$), a natural estimator $\hat{\mu}_w$ with low expected squared-error loss is a shrinkage estimator of the form $\hat{\mu}_w = C Y^*_w$ with scalar $C$, such as the \cite{James:1992jm} estimator for which $C = 1 - \frac{(k - 2) \| Y^*_r \|^2}{(n - m - k + 1) \| Y^*_w \|^2}$ (or its positive part).
    While improving with respect to expected squared-error loss ($a'a = \text{const.} \cdot \I_k$), this specific estimator may yield higher (conditional) expected loss in $\mu_x$ when the implied loss function for $\mu_w$ deviates from squared-error loss ($a'a \neq \text{const.} \cdot \I_k$, so the loss function is not invariant under rotations).
    We will show below that it is still appropriate in the case of independence of treatment and control.
    
    \subsection{From conditional to unconditional loss}
    \label{sect:uncond}

    For conditional inference it is known that the least-squares estimator is admissible for estimating $\beta$ provided $m \leq 2$ and inadmissible provided $m \geq 3$ no matter what the dimensionality $k$ of the nuisance parameter $\gamma$ is \citep{James:1992jm}, as the rank of the loss function is decisive.
    The above construction does not provide a counter-example to this result:
    the rank of $a'a =(w' q_w)^{-1} w' q_x q'_x w (q'_w w)^{-1}$ is at most $m$, so for $m \leq 2$, $\hat{\mu}_w = Y^*_w$ remains admissible for the loss function on the right.
    While we could achieve improvements for $m \geq 3$ -- through shrinkage in $\hat{\mu}_w$ and/or directly in $\hat{\mu}_x$ -- our interest is in the case where $m$ is low and $k$ is high.
    Conditional on $X{=}x, W{=}w$ we can thus not hope to achieve improvements that hold for any $(\beta,\gamma)$,
    but we can still hope that shrinkage estimation of $\mu_w$ yields better estimates of $\beta$ on average over draws of the data.
    
    To this end, assume that
    \begin{align*}
        \vec(W)|X{=}x \sim \N(\vec(\1 \alpha_W + x \beta_W),\Sigma_W \otimes \I_n)
    \end{align*}
    (that is, $W_i | X{=}x \stackrel{\text{iid}}{\sim} \N(1 \alpha_W + x_i \beta_W,\Sigma_W)$).
    Here, $\Sigma_W \in \R^{k \times k}$ is symmetric positive-definite (but not necessarily known).
    $\alpha_W \in \R^{1 \times k}, \beta_W  \in \R^{m \times k}$ describe the conditional expectation of control variables given the regressors $X{=}x$.
    The case where $x$ and $W$ are orthogonal ($\beta_W = \O_{m \times k}$) and controls $W$ thus not required for identification will play a special role below.
    
    Given $X{=}x$, assume $(q_{\1},q_x)$ is deterministic, and fix $q_{\perp}$ such that $\tilde{q} = (q_{\1},q_x,q_{\perp}) \in \R^{n \times n}$ is orthonormal.
    Note that
    \begin{align*}
        \vec((q_x,q_{\perp})' W) | X{=}x
        \sim \N\left(\vec\left(
            \begin{pmatrix}
                q'_x x \beta_W \\
                \O_{n - 1 - m \times k}
            \end{pmatrix}
        \right),\Sigma_W \otimes \I_{n-1}\right).
    \end{align*}
    In particular, $q_x' W \indep q_\perp' W$.
    It follows with
    \begin{align*}
        (q_x,q_{\perp})' Y | X{=}x, W{=}w &\sim \N\left(
        \begin{pmatrix}
            q'_{x} x \beta + q'_{x} w \gamma \\
            q'_{\perp} w \gamma
        \end{pmatrix},
        \sigma^2 \I_{n - 1}
        \right)
    \end{align*}
    that indeed $q_x' (Y,W) \indep q_\perp' (Y,W)$.
    
    Conditional on $W{=}w$ in the above derivation,
    $
        a \hat{\mu}_w - a \mu_w
        = q'_x w \hat{\gamma} - q'_x w \gamma
    $
    for $\hat{\gamma} = (q'_w w)^{-1} \hat{\mu}_w$ a function of $q'_\perp w$ and $(Y^*_w,Y^*_r) = (q'_w q_\perp,q'_r q_\perp) (q'_\perp Y)$,
    so $\hat{\gamma} = \hat{\gamma}(q'_\perp Y,q'_\perp w)$.
    Assuming measurability, $\hat{\gamma}(q'_\perp Y,q'_\perp W) \indep (q'_x Y, q'_x W)$.
    Now writing $\hat{\gamma} = \hat{\gamma}(q'_\perp y,q'_\perp w)$ this implies that
    \begin{align*}
        \E[\| \hat{\mu}_w - \mu_w \|_{a'a}^2 |X{=}x,q'_\perp(Y,W) = q'_\perp (y,w)]
        &= \| \hat{\gamma} - \gamma \|^2_{\E[ W' q_x q'_x W|X{=}x]}
    \end{align*}
    with
    $
        \E[ W' q_x q'_x W|X{=}x]
        = \beta'_W x' q_x q'_x x \beta_W + m \Sigma_W
    $
    of full rank $k$.
    For the expectation of the implied $\hat{\beta}$,
    we find
    \begin{align*}
        \E[\hat{\beta} |X{=}x,q'_\perp(Y,W) = q'_\perp (y,w)]
        = \beta - \beta_W (\hat{\gamma} - \gamma).
    \end{align*}
    We obtain the following characterization of conditional bias and squared-error loss of the implied estimator $\hat{\beta}$:
    \begin{lem}[Properties of the two-step estimator]
        \label{lem:twostepprediction}
        Let $(\tilde{Y},\tilde{W})$ be jointly distributed according as
        \begin{align*}
            \vec(\tilde{W}) &\sim \N(\0_{k (n - 1 - m)},\Sigma_W \otimes \I_{n-1-m}),
            \\
            \tilde{Y} | \tilde{W}=\tilde{w} &\sim
            \N(\tilde{w} \gamma, \sigma^2 \I_{n - 1 - m}),
        \end{align*}
        and write $\tilde{\E}$ for the corresponding expectation operator.
        For any measurable estimator
        $
            \hat{\gamma}: \R^{n - m - 1} \times \R^{n - m - 1 \times k} \rightarrow \R^k
        $
        with $\tilde{\E}[ \|\hat{\gamma}(\tilde{Y},\tilde{W})\|^2] < \infty$,
        the estimator
        $
            \hat{\beta}(y,w) = (q'_x x)^{-1} q'_x y - (q'_x x)^{-1} q'_x w \hat{\gamma}(q'_\perp y,q'_\perp w)
        $
        defined for convenience for fixed $x$,
        has conditional bias
        \begin{align*}
            \E[\hat{\beta}(Y,W) |X{=}x] - \beta
            = - \beta_W (\tilde{\E}[\hat{\gamma}(\tilde{Y},\tilde{W})] - \gamma)
        \end{align*}
        and expected (prediction-norm) loss
        \begin{align*}
            \E[\|\hat{\beta}(Y,W) - \beta \|^2_{x' q_x q'_x x} |X{=}x]
            = m \sigma^2 + \tilde{\E}[\|\hat{\gamma}(\tilde{Y},\tilde{W}) - \gamma\|^2_\phi]
        \end{align*}        
        for $\phi = \beta'_W x'q_x q_x' x \beta_W + m \Sigma_W$.
        
    \end{lem}
    Note that this lemma does not rely on $n \geq 1 + m + k$, and indeed generalizes to the case $n > 1 + m$ for any $k \geq 1$, including $k > n$.
    
    \subsection{Exogenous treatment}
    
    We consider the special case where treatment is exogenous, and thus $\beta_W = \O_{m \times k}$.
    This assumption could be justified, for example, in a randomized trial.
    Note that in this case in addition to the linear least-squares estimator in the ``long'' regression that includes controls $W$ another natural unbiased (conditional on $X{=}x$) estimator is available, namely the coefficient $(q'_x x)^{-1} q'_x Y$ in the ``short'' regression without controls.
    The ``long'' and ``short'' regression represent special (edge) cases in the class of two-step estimators introduced above, which are all unbiased in that sense under the exogeneity assumption:
    
    \begin{cor}[A class of unbiased two-step estimators]
        If $\beta_W = \O_{m \times k}$ then for any $\hat{\gamma}$ and $\hat{\beta}$ as in \autoref{lem:twostepprediction} $\E[\hat{\beta}(Y,W) |X{=}x] = \beta$.
        Furthermore,
        \begin{align*}
            \E[\|\hat{\beta}(Y,W) - \beta \|^2_{x' q_x q'_x x} |X{=}x]
            = m \tilde{\E}[(\tilde{Y}_0 - \tilde{W}_0' \hat{\gamma}((\tilde{Y}_i,\tilde{W}_i)_{i=1}^{n -1 - m}))^2]
        \end{align*}
        for $(\tilde{Y}_i,\tilde{W}_i)_{i=0}^{n -1 - m}$ iid with $\tilde{W}_i \sim \N(\0_k,\Sigma_W)$, $\tilde{Y}_i|\tilde{W}_i=\tilde{w}_i \sim \N(\tilde{w}'_i \gamma,\sigma^2)$ (here, $(\tilde{Y}_i,\tilde{W}_i)_{i=1}^{n -1 - m}$ is the training sample and $(\tilde{Y}_0,\tilde{W}_0)$ an additional test point drawn from the same distribution).
    \end{cor}
    
    This corollary clarifies that the class of natural estimators derived above are unbiased conditional on $X{=}x$ (but not necessarily on $X{=}x,W{=}w$ jointly), with expected loss equal to the expected out-of-sample prediction loss in a prediction problem where the prediction function $\tilde{w}_0 \mapsto \tilde{w}'_0 \hat{\gamma}$ is trained on $n-1-m$ iid draws, and evaluated on an additional, independent draw $(\tilde{Y}_0,\tilde{W}_0)$ from the same distribution.
    The ``long'' and ``short'' regressions are included as the special cases $\hat{\gamma}(\tilde{w},\tilde{y}) = (\tilde{w}'\tilde{w})^{-1} \tilde{w}' \tilde{y}$ and $\hat{\gamma} \equiv \0_k$, respectively.
    
    The covariates in training and test sample follow the same distribution,
    which suggests an estimator that is invariant to rotations in the corresponding $k$-means problem.
    Indeed, the dominating estimator I construct in the following results is of the form
    \begin{align*}
        \hat{\mu}_w = \left( 1 - \frac{p \|Y^*_r \|^2}{\|Y^*_w \|^2}  \right) Y^*_w,
    \end{align*}
    where the standard \cite{James:1992jm} estimator (for unnknown $\sigma^2$) is recovered at $p = \frac{k - 2}{n - m - k + 1}$.

    \begin{thm}[Inadmissibility of OLS among unbiased estimators]
        \label{thm:inadmiss}
        Maintain $\beta_W = \O_{m \times k}$.
        Denote by $(\hat{\alpha}^{\OLS},\hat{\beta}^\OLS,\hat{\gamma}^\OLS)$ the coefficients and by $\SSR = \|Y - \1 \hat{\alpha}^{\OLS} - X \hat{\beta}^\OLS - W \hat{\gamma}^\OLS \|^2$ the sum of squared residuals in a linear least-squares regression of $Y$ on an $\1$, $X$, and $W$.
        Write $h = \I_n - \1_n \1_n'/n$ (the annihilator matrix with respect to the intercept).
        Assume that $k \geq 3$ and $n \geq m + k + 2$.
        Then, the two-step estimator $\hat{\beta} = (X' h X)^{-1} X' h (Y - W \hat{\gamma})$
        with
        \begin{align*}
            \hat{\gamma} = \left( 1 - \frac{p \: \SSR}{\| \hat{\gamma}^\OLS \|_{W' h (\I - X (X'hX)^{-1} X') h W}^2}  \right) \hat{\gamma}^{\OLS}
        \end{align*}
        where $p \in \left(0,\frac{2 (k - 2)}{n - m - k + 2}\right)$
        is unbiased for $\beta$ given $X{=}x$ and dominates $\hat{\beta}^\OLS$ in the sense that
        \begin{align*}
            \E[\|\hat{\beta} - \beta \|^2_{X' h X} |X{=}x]
            < \E[\|\hat{\beta}^\OLS - \beta \|^2_{X' h X} |X{=}x].
        \end{align*}
    \end{thm}
    
    \begin{proof}
        The OLS estimator in the theorem corresponds to $\hat{\gamma}^{\OLS}(\tilde{y},\tilde{w}) = (\tilde{w}'\tilde{w})^{-1} \tilde{y}'\tilde{w}$ in \autoref{lem:twostepprediction},
        which yields the maximum-likelihood estimator $\hat{\gamma}^{\OLS}(\tilde{Y},\tilde{W})$ for $\gamma$ given data
        \begin{align*}
            \vec(\tilde{W}) &\sim \N(\0_{k (n - 1 - m)},\Sigma_W \otimes \I_{n-1-m}),
            \\
            \tilde{Y} | \tilde{W}=\tilde{w} &\sim
            \N(\tilde{w} \gamma, \sigma^2 \I_{n - 1 - m}).
        \end{align*}
        By \cite{Baranchik:1973eb}, this maximum-likelihood estimator is inadmissible with respect to the risk
        $
            \tilde{\E}[\|\hat{\gamma} - \gamma \|^2_{\Sigma_W}]
        $
        and thus for $\tilde{\E}[\|\hat{\gamma} - \gamma \|^2_{\phi}$ in \autoref{lem:twostepprediction}, as $\phi = m \Sigma_W$ for $\beta_W = \O_{m \times k}$.
        However, \cite{Baranchik:1973eb} also includes an intercept that is estimated, but does not enter the loss function.
        To formally use the result for our case without intercept in the first-step prediction exercise,
        I construct an augmented problem such that the dominance result in the augmented problem implies the theorem.
        
        To this end,
        let 
        \begin{align*}
            \vec(W^a) &\sim \N(\0_{k (n  - m)},\Sigma_W \otimes \I_{n-m}),
            \\
            Y^a | W^a=w^a &\sim
            \N(w^a \gamma, \sigma^2 \I_{n  - m}).
        \end{align*}
        (which has one additional sample point, and could without loss include intercepts in $W^a$, $Y^a$).
        By \citet[Theorem~1]{Baranchik:1973eb},
        the estimator
        \begin{align*}
            \hat{\gamma}^a = \left(1 - p \frac{(Y^a)' h^a Y^a - \| \hat{\gamma}^{a,\OLS} \|^2_{(W^a)' h^a W^a}}{\| \hat{\gamma}^{a,\OLS} \|^2_{(W^a)' h^a W^a}} \right) \hat{\gamma}^{a,\OLS}
        \end{align*}
        strictly dominates $\hat{\gamma}^{a,\OLS} = ((W^a)' h^a W^a)^{-1} (W^a)' h^a Z^a$,
        where $h^a = \I_{n - m} - \1_{n - m} \1'_{n - m} / (n - m)$,
        in the sense that
        \begin{align*}
            \E^a[(\hat{\gamma}^a - \gamma)' \Sigma_W (\hat{\gamma}^a - \gamma)]
            < \E^a[(\hat{\gamma}^{a,\OLS} - \gamma)' \Sigma_W (\hat{\gamma}^{a,\OLS} - \gamma)]
        \end{align*}
        for any $\gamma \in \R^k$,
        provided that $p \in \left( 0, \frac{2 (k - 2)}{n - m - k + 2}\right)$ with $k \geq 3$ and $n - m \geq k + 2$.
        
        We now show that this implies dominance of $\hat{\gamma}(\tilde{Y},\tilde{W})$ for
        \begin{align*}
            \hat{\gamma}(\tilde{y},\tilde{w}) = \left(1 - p \frac{\tilde{y}'\tilde{y} - \| \hat{\gamma}^{\OLS}(\tilde{y},\tilde{w}) \|^2_{\tilde{w}'\tilde{w}}}{\| \hat{\gamma}^{\OLS}(\tilde{y},\tilde{w}) \|^2_{\tilde{w}'\tilde{w}}} \right) \hat{\gamma}^{\OLS}(\tilde{y},\tilde{w})
        \end{align*}
        in the original problem.
        Let $q^a \in \R^{(n - m) \times (n - m - 1)}$ be such that $(q^a,\1_{n - m}/(n-m))$ is orthonormal (that is, the columns of $q^a$ complete $\1_{n - m}/(n-m)$ to an orthonormal basis of $\R^{m - n}$).
        This implies that $q^a (q^a)' = h^a$ and $(q^a)' q^a = \I_{n - m - 1}$.
        Then, $(q^a)' (Y^a,W^a) \stackrel{d}{=} (\tilde{Y},\tilde{W})$.
        In particular,
        \begin{align*}
            ((Y^a)'h^a (Y^a), (Y^a)'h^a (W^a), (W^a)'h^a (W^a)) \stackrel{d}{=} (\tilde{Y}'\tilde{Y},\tilde{Y}'\tilde{W},\tilde{W}'\tilde{W})
        \end{align*}
        and thus $(\hat{\gamma}^a,\hat{\gamma}^{a,\OLS}) \stackrel{d}{=} (\hat{\gamma}(\tilde{Y},\tilde{W}),\hat{\gamma}^{\OLS}(\tilde{Y},\tilde{W}))$.
        We have thus established
        \begin{align*}
            &\tilde{\E}[(\hat{\gamma}(\tilde{Y},\tilde{W}) - \gamma)' \Sigma_W (\hat{\gamma}(\tilde{Y},\tilde{W}) - \gamma)] \\
            &< \tilde{\E}[(\hat{\gamma}^{\OLS}(\tilde{Y},\tilde{W}) - \gamma)' \Sigma_W (\hat{\gamma}^{\OLS}(\tilde{Y},\tilde{W}) - \gamma)].
        \end{align*}
        Note that
        $\hat{\gamma}^{\OLS}(\tilde{y},\tilde{w}) = (\tilde{w}'\tilde{w})^{-1} \tilde{y}'\tilde{w}$ in \autoref{lem:twostepprediction} does indeed yield $\hat{\gamma}^{\OLS}$ and $\hat{\beta}^{\OLS}$ in the theorem,
        and that this extends to $\hat{\gamma}$ and $\hat{\beta}$ by
        \begin{align*}
            \hat{\gamma}(q'_\perp y,q'_\perp w)
            = \left( 1 - p \frac{ \| y \|^2_{q_{\perp} q'_{\perp}} - \|  \hat{\gamma}^{\OLS}(\ldots) \|^2_{w ' q_{\perp} q'_{\perp} w}}{\|  \hat{\gamma}^{\OLS}(\ldots) \|^2_{w ' q_{\perp} q'_{\perp} w}} \right)\hat{\gamma}^{\OLS}(\ldots)
        \end{align*}
        with
        $q_{\perp} q'_{\perp} = h (\I - x (x'hx)^{-1} x') h$
        and
        \begin{align*}
            \SSR &= \|Y - \1 \hat{\alpha}^{\OLS} - X \hat{\beta}^\OLS - W \hat{\gamma}^\OLS \|^2
            \\
            &= \|Y - W \hat{\gamma}^\OLS \|_{h (\I - X (X'hX)^{-1} X') h}^2 \\
            &= \|Y \|_{h (\I - X (X'hX)^{-1} X') h}^2 - \| W \hat{\gamma}^\OLS \|_{h (\I - X (X'hX)^{-1} X') h}^2 \\
            &= \|Y \|_{h (\I - X (X'hX)^{-1} X') h}^2 - \| \hat{\gamma}^\OLS \|_{W 'h (\I - X (X'hX)^{-1} X') h W}^2.
        \end{align*}
        Unbiasedness and dominance follow with $\beta_W = \O_{m \times k}$ in \autoref{lem:twostepprediction}.
    \end{proof}
    
    Note that the result extends to the positive-part analog for which the shrinkage factor is set to zero whenever the expression is negative.
    For $m=1$, the following dominance is immediate:
    
    \begin{cor}[A non-contradiction of Gauss--Markov]
        For exogenous treatment, $m=1$, $k \geq 3$, and $n \geq k + 3$,
        there exists an estimator $\hat{\beta}$ with $\E[\hat{\beta}|X{=}x] = \beta$
        and $\Var(\hat{\beta}|X{=}x) < \Var(\hat{\beta}^\OLS|X{=}x)$.
    \end{cor}
        
    The assumption of exogenous treatment is essential for this result, as dropping conditioning on $W$ and restricting interest to $\beta$ would not suffice to break optimality of linear least-squares.

    \section{Invariance Properties and Bayesian Interpretation}
    \label{sect:invariance}
    
    Starting with the transformations in \autoref{sect:uncond}, we consider the decision problem of estimating $\beta$ (equivalently, $\mu_x$).
    Guided by the treatment of a linear panel-data model in \cite{Chamberlain:2009gx},
    I develop the specific estimator proposed in \autoref{thm:inadmiss} as (the empirical Bayes version of) an invariant Bayes estimator with respect to a partially uninformative (improper) Jeffreys prior.

    \subsection{Decision problem set-up}
    
    In this section, we condition on $X$ throughout and assume that covariates $W$ are Normally distributed given $X$.
    Writing $W^*_x = q'_x W, W^*_\perp = q'_\perp W,Y^*_x = q'_x Y, Y^*_\perp = q'_\perp Y$,
    the transformation developed in \autoref{sect:uncond} yields the joint distribution
    \begin{align}
    \label{eqn:canonicaleasy}
    \begin{aligned}
        \begin{pmatrix}
            W^*_x \\ W^*_\perp
        \end{pmatrix}
        &=
        \begin{pmatrix}
            \mu_W \\ \O_{s \times k}
        \end{pmatrix}
        +
        V_W
        \Sigma^{1/2}_W
        &
        (V_W)_{ij} &\stackrel{\text{iid}}{\sim} \N(0,1)
        \\
        \begin{pmatrix}
            Y^*_x \\ Y^*_\perp
        \end{pmatrix}
        &=
        \begin{pmatrix}
            \mu_x + W^*_x \gamma \\ W^*_\perp \gamma 
        \end{pmatrix}
        +
        V_Y \sigma^2
        &
        (V_Y)_{i} &\stackrel{\text{iid}}{\sim} \N(0,1)
    \end{aligned}
    \end{align}
    where $\Sigma_W^{1/2}$ is the unique symmetric positive-definite square-root of the symmetric positive-definite matrix $\Sigma_W$, and $V_W$ and $V_Y$ are independent.
    Here, in addition to $\mu_x = q'_x x \beta$, also  $\mu_W = q'_x x \beta_W$, and $s = n - m - 1$.
    I write $\mathcal{Z} = \R^{m + s} \times \R^{(m + s) \times k}$ for the sample space from which $(Y^*,W^*)$ is drawn according to this $\P_\theta$,
    where I parametrize $\theta = (\mu_x,\gamma) \in \Theta = \R^m \times \R^k$.
    (I take $\sigma^2,\Sigma_W,\mu_W$ to be constants.)
    
    The action space is $\mathcal{A}= \R^m$, from which an estimate of $\mu_x$ is chosen.
    As the loss function $L: \Theta \times \mathcal{A} \rightarrow \R$ I take squared-error loss $L(\theta,a) = \| \mu_x - a \|^2$.
    An estimator $\hat{\beta}: \mathcal{Z} \rightarrow \mathcal{A}$ from the previous section is a feasible decision rule in this decision problem.
    
    \subsection{A set of transformations}
        
        For an element $g = (g_\mu,g_x,g_W,g_\perp)$ in the (product) group $G = \R^m \times O(m) \times O(k) \times O(s)$, where $\R^m$ denotes the group of real numbers with addition (neutral element $0$) and $O(k)$ the group of ortho-normal matrices in $\R^{k \times k}$ with matrix multiplication (neutral element $\I_k$),
        consider the following set of transformations (which are actions of $G$ on $\mathcal{Z}, \Theta, \mathcal{A}$):
        \begin{itemize}
            \item Sample space: $m_\mathcal{Z}: G \times \mathcal{Z} \rightarrow \mathcal{Z}$,
                            \begin{align*}
                                &(g,(y_x,y_\perp,w_x,w_\perp)) \\
                                &\mapsto
                                (
                                    g_x y_x + g_\mu, g_\perp y_\perp,
                                    g_x w_x \Sigma_W^{-1/2} g'_W \Sigma_W^{1/2}, g_\perp w_\perp \Sigma_W^{-1/2} g'_W \Sigma_W^{1/2}
                                )
                            \end{align*}
            \item Parameter space: $m_\Theta: G \times \Theta \rightarrow \Theta$,
                    \begin{align*}
                        (g,(\mu_x,\gamma)) \mapsto
                        (g_x \mu_x + g_\mu,\Sigma_W^{-1/2} g_W \Sigma_W^{1/2} \gamma)
                    \end{align*}
            \item Action space: $m_\mathcal{A}: G \times \mathcal{A} \rightarrow \mathcal{A}, (g,a) \mapsto g_x a + g_\mu$
        \end{itemize}
        For exogenous treatment, these transformations are tied together by leaving model and loss invariant. Indeed, the following is immediate from \autoref{eqn:canonicaleasy}:%
        \footnote{Alternatively, we could have treated $\mu_W$ as an element of the parameter space and extend the analysis to the case of endogenous treatment. Adding $(g,\mu_W) \mapsto g_x \mu_W \Sigma_W^{-1/2} g'_W \Sigma_W^{1/2}$ to the action on the parameter space would have retained invariance.}
        
        \begin{prop}[Invariance of model and loss]
        \label{prop:invariance}
        For $\mu_W = \O_{m \times k}$:
            \begin{enumerate}
                \item The model is invariant: $m_\mathcal{Z}(g,(Y^*,W^*)) \sim \P_{m_\Theta(g,\theta)}$ for all $g \in G$.
                \item The loss is invariant: $L(m_\Theta(g,\theta),m_\mathcal{A}(g,a)) = L(\theta,a)$ for all $g \in G$.
            \end{enumerate}
        \end{prop}

    \subsection{An invariant Bayes estimator \ldots}
    
    By Proposition~\ref{prop:invariance}, a natural (generalized) Bayes estimator of $\mu_x$ is derived from an improper prior on $\theta$ that is invariant under the action of $G$ on $\Theta$,
    as this will yield a decision rule $d: \mathcal{Z} \rightarrow \mathcal{A}$ that is invariant in the sense that $d(m_\mathcal{Z}(g,(y,w))) = m_\mathcal{A}(g,d((y,w)))$ for all $(g,(y,w)) \in G \times \mathcal{Z}$.
    This implies for $\mu_x$ as an improper prior the Haar measure with respect to the translation action (i.e. up to a multiplicative constant the $\sigma$-finite Lebesgue measure on $\mathbb{R}^m$),
    and for $\gamma$ a prior that is uniform on ellipsoids $\gamma' \Sigma_W \gamma = \omega$.
    Taking $\frac{\omega}{\tau^2} \sim \chi^2_m$ with some $\tau > 0$ yields the prior $\gamma \sim \N(\0,\tau^2 \Sigma_W^{-1})$.
    With a product prior for $\theta$, the resulting generalized Bayes estimator for $\mu_x$ -- which minimizes posterior loss conditional on the data -- is
    \begin{align*}
        \E[\mu_x | Y^*=y, W^*=w]
        &= y_x - w_x \E[\gamma | Y^*=y, W^*=w] \\
        &= y_x - w_x (w'_\perp w_\perp + \sigma^2 \Sigma_W / \tau^2)^{-1} w'_\perp y_\perp.
    \end{align*}
    
    \subsection{\ldots and a specific empirical Bayes implementation}
    
    Replacing $\Sigma_W$ by the specific sample analog $W_\perp' W_\perp / s$,
    we obtain
    the estimator $Y^*_x - \frac{s \tau^2}{s \tau^2 + \sigma^2} W^*_x ((W^*_\perp)' W^*_\perp)^{-1} (W^*_\perp)' Y^*_\perp$.
    Similarly assuming that $\gamma \sim \N(\0,s \tau^2 W_\perp' W_\perp)$,
    an unbiased estimator of $\frac{s \tau^2}{s \tau^2 + \sigma^2}$ (given $W$)
    is
    \begin{align*}
        C = 1 - \frac{(Y^*_\perp)' (\I_s - W^*_\perp ((W^*_\perp)' W^*_\perp)^{-1} (W^*_\perp)') Y^*_\perp / (s - k)}{(Y^*_\perp)' W^*_\perp ((W^*_\perp)' W^*_\perp)^{-1} (W^*_\perp)' Y^*_\perp / (k - 2)}.
    \end{align*}
    This estimator corresponds to the estimator from \autoref{thm:inadmiss} at $p = \frac{k - 2}{s - k} = \frac{k - 2}{n - m - k - 1}$.
    By construction, it retains the invariance of the associated generalized Bayes estimator. 
    This is not specific to this value of $p$:
    
        \begin{prop}[Invariance of estimator]
        \label{prop:invarianceest}
            For any $p$, the estimator $\hat{\beta}$ from \autoref{thm:inadmiss} is invariant with respect to the above actions of $G$.
        \end{prop}

    \section*{Conclusion}

    A natural application of James--Stein shrinkage to control variables in a Normal linear model consistently reduces expected prediction error without introducing bias in the treatment parameter of interest provided treatment is random.
    In this case, the linear least-squares estimator is thus inadmissible even among unbiased estimators.
    
    In a companion paper \citep{JSIV}, I show how shrinkage in at least four instrumental variables in a canonical structural form provides consistent bias improvement over the two-stage least-squares estimator.
    Together, these results suggests different roles of overfitting in control and instrumental variable coefficients, respectively: while overfitting to control variables induces variance, overfitting to instrumental variables in the first stage of a two-stage least-squares procedure induces bias.
    
    \bibliography{Bibliography}	
    
\end{document}

%% file: Headers/Header.tex
    \usepackage[T1]{fontenc}
	\usepackage[utf8]{inputenc}
	\usepackage[english]{babel}
	\usepackage{amsfonts}
	\usepackage{amssymb}
	\usepackage{enumitem}
	\usepackage{bbm}
	\usepackage{mathtools}
	\usepackage{subfigure}
	\usepackage{booktabs}
	\usepackage{graphicx}
	
	\usepackage{soul}
	\usepackage{xspace}

    \usepackage{titlesec}

    \titleformat{\section}
    {\centering\normalfont\scshape}{\thesection}{1em}{}
    
    \titleformat{\subsection}
    {\centering\normalfont\itshape}{\thesubsection}{1em}{}
	
	\usepackage{setspace}
	\usepackage{natbib}
	\usepackage{hyperref}
	\usepackage{cleveref}

	\usepackage[framemethod=tikz]{mdframed}
	\usepackage{todonotes}
	\presetkeys{todonotes}{color=black!5,inline}{} 

	\newcommand{\rf}[1]{\comment{Reference: \url{#1}}}

    \usepackage{tcolorbox}
    \newtcbox{\feedback}{nobeforeafter,colframe=black,colback=white,boxrule=0.5pt,arc=2pt,
      boxsep=0pt,left=2pt,right=2pt,top=2pt,bottom=2pt,tcbox raise base}
      
    \usepackage{amsthm}
     
    \newtheorem{thm}{Theorem}

    \newtheorem{prop}{Proposition}
    \newtheorem{lem}{Lemma}
    \newtheorem{cor}{Corollary}
    
    \addto\extrasenglish{%
    }

\newcommand\blfootnote[1]{%
  \begingroup
  \renewcommand\thefootnote{}\footnote{#1}%
  \addtocounter{footnote}{-1}%
  \endgroup
}

%% file: Headers/Notation.tex
	\renewcommand{\P}{\mathop{}\!\textnormal{P}}
	\newcommand{\E}{\mathop{}\!\textnormal{E}}

	\newcommand{\Var}{\mathop{}\!\textnormal{Var}}
	
	\newcommand{\N}{\mathcal{N}}

	\newcommand{\0}{\mathbf{0}}
	\newcommand{\1}{\mathbf{1}}
	\newcommand{\I}{\mathbb{I}}
	\renewcommand{\O}{\mathbb{O}}
	\renewcommand{\vec}{\textnormal{vec}}

	\newcommand{\R}{\mathbb{R}}

    \newcommand{\indep}{\rotatebox[origin=c]{90}{$\models$}}
    
	\newcommand{\OLS}{\textnormal{OLS}}  
	
    \newcommand{\SSR}{\textnormal{SSR}}  